\documentclass[9pt]{amsart}
\usepackage{amsmath, amssymb}
\usepackage{amsfonts,amscd}
\usepackage[a4paper, margin=1.25in]{geometry}
\usepackage{parskip}

\allowdisplaybreaks

\newtheorem{theorem}{Theorem}
\newtheorem{defn}[theorem]{Definition}
\newtheorem{prop}[theorem]{Proposition}
\renewenvironment{proof}{\par\noindent{\bf Proof.}}{$\square$\par\bigskip}
\newtheorem{lemma}[theorem]{Lemma}
\newtheorem{remark}[theorem]{Remark}
\newtheorem{cor}[theorem]{Corollary}
\newtheorem{conjecture}[theorem]{Conjecture}
\newtheorem{question}[theorem]{Question}

\newtheorem{thm}{Theorem}

\newtheorem*{unmthm}{Theorem}

\def\Z{\mathbb Z}

\def\Q{\mathbb Q}
\def\R{\mathbb R}
\def\C{\mathbb C}

\def\H{\mathbb H}

\def\mod{\operatorname{mod}}

\def\uk{\underline{k}}

\def\O{\operatorname{O}}
\def\o{\operatorname{o}}
\def\log{\operatorname{log}}

\def\mod{\rm mod \, }

\def\thtp{\operatorname{\theta_{\textit{f}}\,(\textit{p})}}

\def\thtpone{\operatorname{\theta_{\textit{f}_1}(\textit{p})}}
\def\thtptwo{\operatorname{\theta_{\textit{f}_2}(\textit{p})}}
\def\thtpi{\operatorname{\theta_{\textit{f}_i}(\textit{p})}}
\def\thtpione{\operatorname{\theta_{\textit{f}_{i+1}}(\textit{p})}}

\def\thtq{\operatorname{\theta_{\textit{f}}\,(\textit{q})}}

\def\Lf{\operatorname{\mathcal L_{\textit{f}}}}
\def\Lpi{\operatorname{\mathcal L_{\pi}}}
\def\Lj{{\mathcal L_{j}}}

\def\V{\operatorname{\mathrm{vol}(\Gamma)}}
\def\gp{\mathfrak p}
\def\gq{\mathfrak q}

\def\SL{\mathrm{SL}}

\begin{document}

\title{Pair correlation statistics for Sato-Tate sequences}
\author[Baskar Balasubramanyam]{Baskar Balasubramanyam}
\date{\today}
\address{Baskar Balasubramanyam, IISER Pune, Dr Homi Bhabha Road, Pashan, Pune - 411008, Maharashtra, India}
\email{baskar@iiserpune.ac.in}
\author[Kaneenika Sinha]{Kaneenika Sinha}
\address{Kaneenika Sinha, IISER Pune, Dr Homi Bhabha Road, Pashan, Pune - 411008, Maharashtra, India}
\email{kaneenika@iiserpune.ac.in}

\maketitle

\begin{abstract} We investigate the pair correlation statistics for sequences arising from Hecke eigenvalues with respect to spaces of primitive modular cusp forms.  We derive the average pair correlation function of Hecke angles lying in small subintervals of $[0,1]$.  The averaging is done over non-CM newforms of weight $k$ with respect to $\Gamma_0(N).$  We also derive similar statistics for  Hilbert modular forms and modular forms on hyperbolic 3-spaces.
\end{abstract}
\bigskip

\section{Introduction}\label{classical}

Let $k$ and $N$ be positive integers with $k$ even.  Let  $S(N,k)$ denote the space of modular cusp forms of weight $k$ with respect to $\Gamma_0(N).$  For $n \geq 1,$ let $T_n$ denote the $n$-th Hecke operator acting on $S(N,k).$  We denote the set of Hecke newforms  in $S(N,k)$  by $\mathcal F_{N,k}.$  Any $f(z) \in \mathcal F_{N,k}$ has a Fourier expansion
$$f(z) = \sum_{n=1}^{\infty} {n^{\frac{k-1}{2}}}a_f(n) q^n, \qquad q = e^{2\pi i z},$$
where 
$a_f(1) = 1$ and
$$\frac{T_n(f(z))}{n^{\frac{k-1}{2}}} = a_f(n) f(z),\,n \geq 1.$$
Let us fix $N$ and $k$ and consider a newform $f(z)$ in $\mathcal F_{N,k}.$  Let $p$ be a prime number with $(p,N) = 1.$  By a theorem of Deligne, the eigenvalues $a_f(p)$ lie in the interval $[-2,2].$  Denoting $a_f(p) = 2\cos \pi \thtp,\,$ with $ \thtp \in [0,1],$ we consider the Sato-Tate sequence
\begin{equation}\label{ST-sequence}
\{\thtp : \,p \text{ prime},\,(p,N) = 1,\,p \to \infty\} \subseteq [0,1].
\end{equation}

The famous Sato-Tate conjecture, proved by Barnet-Lamb, Geraghty, Harris and Taylor \cite{BGHT} in 2011, is the assertion that if $f$ is a non-CM  in $\mathcal F_{N,k},$ the above sequence is equidistributed in the interval $[0,1]$ with respect to the measure $\mu_\infty (t) dt$, where $\mu_{\infty}(t) = 2\sin^2(\pi t)$.
That is, for any interval $[a,b] \subset [0,1],$
$$\lim_{x \to \infty}\frac{1}{\pi_N(x)}\#\{p \leq x :\, (p,N) = 1,\,\thtp \in [a,b]\} = \int_a^b \mu_{\infty}(t) dt,$$
where $\pi_N(x)$ denotes the number of primes $p \leq x$ such that $(p,N)=1.$ 

The above theorem, considered among the most important mathematical breakthroughs in recent times, has several interesting interpretations.  Let $e(x) = e^{2 \pi i x}.$  By Weyl's classical equidistribution theory that relates equidistribution phenomena to the behaviour of exponential sums (see \cite[Section 7]{KN}), it is equivalent to the assertion that for $m \in \Z,$  the Weyl limits
$$C_m := \lim_{x \to \infty} \frac{1}{\pi(x)} \sum_{p \leq x \atop {(p,N) = 1}} e(m \thtp)$$
are equal to
\begin{equation}\label{limits}
\begin{cases}
1 &\text{ if }m = 0\\
-\frac{1}{2} &\text { if }m = \pm 1\\
0 &\text{ otherwise}.
\end{cases}
\end{equation}
A careful estimation of the Weyl sums $\sum_{n=1}^N e(mx_n)$ associated to a sequence $\{x_n\}$ often helps in obtaining error terms in equidistribution results.  In the case of the Sato-Tate sequence, the sums $\sum_{p \leq x} e(m \thtp)$ and the error terms in the Sato-Tate distribution are closely related to the analytic properties of symmetric power $L$-functions associated to $f$.  In this direction, for $k = 2$ and $N$ squarefree, Murty \cite{Murty} obtained explicit conditional estimates.  Under the assumption that all symmetric power $L$-functions for $f$ can be analytically continued to $\C,$  have suitable functional equations and satisfy the Generalized Riemann Hypothesis, he showed that for an interval $[a,b] \subset [0,1],$
$$\#\{p \leq x :\, (p,N) = 1,\,\thtp \in [a,b]\} = \pi_N(x)\int_a^b \mu_{\infty}(t) dt + O\left(x^{3/4}\sqrt{\log Nx}\right).$$
This explicit error term has been sharpened and generalized to all even $k \geq 2$ by Rouse and Thorner \cite{RT} under the assumption of similar analytic hypotheses.  

In the 1990s, another  distribution aspect of the Hecke angles $\{\thtp\}_{(p,N) = 1}$ was considered, namely their level spacing statistics. Assuming the Sato-Tate conjecture, one ``straightens out" the Sato-Tate sequence by defining
$$H(\thtp) := \int_{0}^{\thtp} 2\sin^2\pi t\, dt = \thtp - \frac{\sin 2\pi \thtp}{2\pi}.$$
The bijection $H$ on $[0,1]$  takes the distribution measure $\mu_{\infty}(t) dt$ to the Lebesgue measure $dx$. We arrange the set $\{H(\thtp\}_{p \leq x}$ in ascending order:
\begin{equation*}
 0 \leq H(\theta_f(p))_1 \leq H(\theta_f(p))_2 \dots \leq H(\theta_f(p))_{\pi_N(x)} \leq 1
 \end{equation*}
and consider the consecutive spacings $H(\theta_f(p))_{i+1} - H(\theta_f(p))_i,\,1 \leq i \leq \pi_N(x).$  Katz asked the following question.
\begin{question}[Katz] \label{Katz}
 Is the level spacing distribution of the sequence $\{H(\theta_f(p)\}_{p \to \infty}$ Poissonnian?  That is, for any $[a,b] \subset [0,\infty),$ is the limit
$$\lim_{x \to \infty}\frac{1}{\pi_N(x)}\#\left\{1 \leq i \leq \pi_N(x):\,H(\theta_f(p))_{i+1} - H(\theta_f(p))_i \in \left[\frac{a}{\pi_N(x)},\frac{b}{\pi_N(x)}\right]\right\} = \int_a^b e^{-t} dt?$$
\end{question}
In other words, Katz asked if the distribution of the spacings between straightened Hecke angles is the same as the distribution of spacings among points in a random sequence picked uniformly and independently in the unit interval.  

Katz and Sarnak \cite[Page 9]{KS} also considered a vertical variant of the above problem.  For a fixed prime $p,$ one defines the multisets
$$ A_p(N,k) = \{H(\thtp),\,f \in \mathcal F_{N,k}\} \subseteq [0,1].$$
The multiset $A_p(N,k)$ is then arranged in ascending order as follows:
$$
0 \leq H(\thtpone) \leq H(\thtptwo) \dots \leq H(\theta_{f_r} (p)) \leq 1.
$$
Here, $r = |\mathcal F_{N,k}|.$
Katz and Sarnak consider the level spacings among the multisets $A_p(N,k)$ for $k = 2,\   N \to \infty$ and ask if the level spacing distribution matches that of a sequence of independent and uniform random points on $[0,1]$. More precisely, they ask the following question:

\begin{question} [{\cite[Page 9]{KS}}] 
Let $p$ be a fixed prime.  Let $k=2$ and $N \neq p$ be a prime.  Is the level spacing distribution of the multisets $A_p(N,k)$ Poissonnian as $N \to \infty?$ That is, is it true that for any $[a,b] \subset [0,\infty),$
$$\lim_{N \to \infty \atop {N \text{ prime} \atop {N \neq p}}}\frac{1}{|\mathcal F_{N,k}|}\#\left\{1 \leq i \leq |\mathcal F_{N,k}| - 1:\,H(\thtpione) - H(\thtpi) \in \left[\frac{a}{|\mathcal F_{N,k}|},\frac{b}{|\mathcal F_{N,k}|}\right]\right\} = \int_a^b e^{-t} dt ?$$
\end{question}
As a partial answer to their question, they average over primes and state the following theorem:
\begin{unmthm} [{\cite[Page 9]{KS}}]
For any $[a,b] \subset [0,\infty),$
\begin{align*}
\lim_{x \to \infty \atop {N \to \infty \atop {N \text{ prime }}}}\frac{1}{\pi(x)}\sum_{p \leq x \atop {p \neq N}} \frac{1}{|\mathcal F_{N,k}|} &\#\left\{1 \leq i \leq |\mathcal F_{N,k}| - 1:\,H(\thtpione) - H(\thtpi) \in \left[\frac{a}{|\mathcal F_{N,k}|},\frac{b}{|\mathcal F_{N,k}|}\right]\right\} \\ 
&= \int_a^b e^{-t} dt. 
\end{align*}
\end{unmthm}

The aim of this article is to prove partial results towards Question \ref{Katz}.  Our main result looks at the obverse of the above Theorem.  While the above theorem considers the spacings in the multisets $A_p(N,k)$ and averages them over primes $p,$ we consider spacings in the sets $\{H(\thtp\}_{p \leq x}$ and average them over all the non-CM newforms $f \in \mathcal F_{N,k}.$  

The level spacing distribution function of a sequence is determined, in turn, by correlation functions which look at distribution of {\it unordered} spacings in the sequence. We now define the pair correlation function of a sequence.

\begin{defn} Let $\{x_n\}$ be a sequence in $[0,1].$  For a positive real number $s$ and a positive integer $M,$ we define the function $R_M(s)$ of this sequence as
$$R_M(s) := \frac{1}{M} \#\left \{ 1 \leq i \neq j \leq M:\,|x_i - x_j| \leq \frac{s}{M} \right\}.$$
The {\bf pair correlation function} $R(s)$ of $\{x_n\}$ is defined as
$$R(s) := \lim_{M \to \infty} R_M(s),$$
provided this limit exists.  
\end{defn}

\begin{remark}\label{ls-ps}
It can be shown that if the level spacing distribution of a sequence $\{x_n\}$ is Poissonnian, then its pair correlation function $R(s) = 2s.$ 
\end{remark}

 In this article, we consider, in particular, the pair correlation statistic of $\{H(\thtp)\}_{p \to \infty}.$  The pair correlation function of the sequence $\{H(\thtp)\}$ that we intend to study in this article is defined as follows:
\begin{defn}\label{pcfST}
Let $k$ and $N$ be positive integers, with $k$ even.  As before, for a non-CM Hecke newform $f \in \mathcal F_{N,k},$
let us choose $\thtp \in [0,1]$ such that $a_f(p) = 2 \cos \pi \thtp$ and define
$$H(\thtp) := \int_{0}^{\thtp} 2\sin^2\pi t \,dt = \thtp - \frac{\sin 2\pi \thtp}{2\pi}.$$
For positive real numbers $s$ and $x,$ we define the function $R_{f,x}(s)$ as
$$R_{f,x}(s) := \frac{1}{\pi_N(x)} \#\left \{ 1 \leq p \neq q \leq x:\, (p,N) = (q,N) = 1,\,|H(\thtp) - H(\thtq)|  \leq \frac{s}{\pi_N(x)} \right\}.$$
The pair correlation function of the sequence $\{H(\thtp)\}$ is defined as 
$$R_f(s) := \lim_{x \to \infty} R_{f,x}(s),$$
provided this limit exists.
\end{defn}

\begin{remark}  By Remark \ref{ls-ps}, an affirmative answer to Question \ref{Katz} would imply that $R_f(s) = 2s.$
\end{remark}

Following the philosophy of relating arithmetic distribution questions about sequences to exponential sums arising from them, the study of the pair correlation function $R_f(s)$ entails the estimation of exponential sums
$$ \sum_{p \leq x} e\left(m H(\thtp)\right)
= \sum_{p \leq x} e\left(m \int_0^{\thtp} 2\sin^2\pi t dt\right)
= \sum_{p \leq x} e\left(m \left(\thtp - \frac{\sin 2\pi \thtp}{2 \pi}\right)\right).$$ 
We approach these exponential sums by localizing the angles $\thtp$ to small subintervals of $[0,1].$  Let $\psi$ be a fixed real number such that $0 < \psi < 1.$  We consider the angles $\thtp$ that lie in localized intervals $$\mathcal I_L = \left[\psi - \frac{1}{L},\psi + \frac{1}{L}\right]$$ and note that
$$\int_{\mathcal I_L} 2 \sin^2\pi t dt \sim \frac{2}{L} \cdot 2 \sin^2 (\pi \psi)\text{ as }L \to \infty.$$
Let $\Lf$ denote the number of angles in the small interval $\mathcal I_L,$ that is,
$$\Lf := \Lf\,(x,L,\psi) := \#\left\{p \leq x:\,(p,N) = 1,\,\thtp \in \mathcal I_L \right\}.$$
The mean spacing of the Hecke angles in the interval $\mathcal I_L$ is $\frac{2}{L\Lf}.$  We look at the unordered spacings 
$$\left\{\frac{L\Lf}{2}(\thtp - \thtq),\, p \neq q \leq x,\,(p,N) = (q,N) = 1,\,\thtp,\,\thtq \in \mathcal I_L\right\}$$
and examine if their distribution matches the (Poissonnian) distribution of the spacings of a random sequence.  To do so, we define a localized analogue of the pair correlation function (\ref{pcfST}) as follows.
\begin{defn}\label{pcflocal}
Let $0 < \psi < 1.$  Let $\mathcal I_L$ denote the interval
$$\left[\psi - \frac{1}{L},\psi + \frac{1}{L}\right].$$  For positive real numbers $s$ and $x,$ we define the function
$$
R_{f,x,L}(s) := \frac{1}{\Lf} \#\left \{ 1 \leq p \neq q \leq x:\, \begin{array}{c} (p,N) = (q,N) = 1,\,\thtp,\,\thtq \in \mathcal I_L,\\
|H(\thtp) - H(\thtq)|  \leq \frac{2s}{L\Lf} \end{array} \right\}. 
$$
Choosing $L = L(x)$ to be an increasing function such that $L(x) \to \infty$ as $x \to \infty,$ we define the {\bf local pair correlation function around $\psi$} as
$$R_{f,\psi}(s) := \lim_{x \to \infty}R_{f,x,L}(s).$$
\end{defn}

Henceforth, we assume that $N$ is a prime, unless indicated otherwise.  This assumption is made for technical simplicity.  To keep the notation uncluttered in all the work that follows, the primes $p$ and $q$ under consideration will implicitly be assumed to be coprime to the level $N$ of the newform $f.$  

In this article, we evaluate the expected value of the pair correlation function $R_{f,\psi}(s)$ on averaging over all non-CM newforms in $\mathcal F_{N,k}.$  Since we are choosing $N$ to be prime, all newforms in $\mathcal F_{N,k}$ are non-CM.  The perspective of averaging (described in detail in Section \ref{Thm2-proof}) enables us to draw upon the Eichler-Selberg trace formula to evaluate $\sum_{f \in \mathcal F_{N,k}}\sum_{p \neq q \leq x} e(m(\thtp -\thtq)).$  Indeed, under appropriate conditions, we are able to show that the expected value of the pair correlation function $R_{f,\psi}(s),$ as we vary over $f \in \mathcal F_{N,k},$ is Poissonnian.

We are now ready to state the first theorem of this article.

\begin{thm}\label{Theorem2}
Let $0 < \psi < 1$ and $L = L(x) \approx \log \log x.$  We consider families $\mathcal F_{N,k}$ with prime levels $N = N(x)$ and even weights $k = k(x)$ such that $\frac{\log (kN)}{x} \to \infty$ as $x \to \infty.$  
Then,
$$\lim_{x \to \infty}\frac{1}{|\mathcal F_{N,k}|}\sum_{f \in \mathcal F_{N,k}}R_{f,x,L}(s) = 2s.$$
That is, 
$$\frac{1}{|\mathcal F_{N,k}|}\sum_{f \in \mathcal F_{N,k}}R_{f,\psi}(s) \sim 2s$$
as $\frac{\log (kN)}{x} \to \infty.$

\end{thm}

\begin{remark}
In order to study the pair correlation function for a fixed $f,$  we need to estimate the exponential sums $\sum_{p \neq q \leq x} e(m(\thtp -\thtq)).$  An immediate approach to do so would be to extend the methods of Murty \cite{Murty} and Rouse--Thorner \cite{RT} and relate these sums to the analytic properties of the symmetric power $L$-functions associated to $f.$  Under several strong assumptions regarding symmetric power $L$-functions, such as, that they can be analytically continued to $\C,$  have suitable functional equations and satisfy the Generalized Riemann Hypothesis, one can use the methods of \cite[cf. Proposition 3.3]{RT} to generate estimates for $\sum_{p \neq q \leq x} e(m(\thtp -\thtq)).$  Unfortunately, these (conditional) estimates are not strong enough to derive the pair correlation function $R_{f,\psi}(s).$  We therefore average over all non-CM newforms in $\mathcal F_{N,k}.$  
\end{remark}

The technique of averaging via a trace formula can be generalized to derive similar statistics in the context of Hilbert modular forms (cf. Theorem \ref{hilbert-Theorem2}) and modular forms on hyperbolic 3-spaces (cf. Theorem \ref{Bianchi-main}).  We describe these in the last two sections of this article.

\subsection*{Outline} This article is organized as follows.  

In Section \ref{smooth}, we consider a smooth analogue of the local pair correlation function.  We express the (smoothened) local pair correlation function of the Sato-Tate sequence in terms of the Weyl sums of Hecke angles.   We also use multiplicative relations among the Hecke eigenvalues to simplify these sums.

In Section \ref{Sato-Tate}, we study the local pair correlation function $R_{f,\psi}(s)$ for Hecke angles by averaging over all $f \in \mathcal F_{N,k}.$  We recall the Eichler-Selberg trace formula and related estimates for the traces of Hecke operators acting on spaces of primitive cusp forms.   This formula is an important tool in proving Theorem \ref{Theorem2}.  

In Section \ref{Thm2-proof}, we combine the analytic techniques of Section \ref{smooth} and the Eichler-Selberg trace formula in Section \ref{Sato-Tate} to prove Theorem \ref{Theorem2}. 

In Section \ref{Hilbert}, Theorem \ref{Theorem2} is generalized to the context of Hilbert modular forms in Theorem \ref{hilbert-Theorem2}.  

Finally, in Section \ref{Bianchi}, we study the pair correlation statistics for angles corresponding to Hecke eigenvalues for modular forms on hyperbolic 3-spaces with respect to $\SL_2(\mathcal O_K),$ where $\mathcal O_K$ is the ring of integers of an imaginary quadratic field $K$ with class number 1.  We prove a suitable generalization of previous theorems in Theorem \ref{Bianchi-main}.
\medskip

\subsection*{Acknowledgements} 
We would like to thank Ze'ev Rudnick for valuable inputs and guidance during the preparation of this article.  

The first named author is supported by Science and Engineering Research Board (SERB) grants EMR/2016/000840 and MTR/2017/000114.
\bigskip

\section{Pair correlation statistics via smooth functions}\label{smooth}

As in Definition \ref{pcflocal}, we take $0 < \psi < 1$.  Focusing on local statistics in small intervals around a point $\psi$ reduces the problem of straightening the angles $\thtp$ to a simple rescaling.  Let us denote $A = 2\sin^2 \pi \psi.$  If $\thtp,\,\thtq \in \mathcal I_L = [\psi - \frac{1}{L},\psi + \frac{1}{L}],$ then
$$ H(\thtp) - H(\thtq) = \int_{\thtq}^{\thtp} 2\sin^2 \pi t dt \sim A(\thtp - \thtq) \text{ as }L \to \infty.$$ 
By the Sato-Tate equidistribution theorem,  
\begin{equation}\label{STLf}
\Lf \sim \pi(x)A\frac{2}{L}  \text{ as }x \to \infty.
\end{equation}
We may choose $L = L(x) \to \infty$ as $x \to \infty.$  Thus, as $x \to \infty,$
\begin{align}\label{localizing}
R_{f,x,L}(s) &:=\frac{1}{\Lf}\#\left\{ p \neq q \leq x,\,\thtp,\,\thtq \in \mathcal I_L,\,H(\thtp) - H(\thtq) \in \left[\frac{-2s}{L\Lf},\frac{2s}{L\Lf}\right]\right\}\nonumber\\
&\sim \frac{L}{2A\pi(x)}\#\left\{ p \neq q \leq x,\,\thtp,\,\thtq \in \mathcal I_L,\,A(\thtp - \thtq) \in \left[\frac{-s}{A \pi(x)},\frac{s}{A \pi(x)}\right]\right\}\\
&\sim \frac{L}{2A\pi(x)}\#\left\{ p \neq q \leq x,\,\thtp,\,\thtq \in\mathcal I_L,\,\thtp - \thtq \in \left[\frac{-s}{A^2 \pi(x)},\frac{s}{A^2 \pi(x)}\right]\right\}\nonumber.
\end{align}
Thus, the treatment of the spacings $H(\thtp) - H(\thtq)$ is equivalent to considering the spacings $\thtp - \thtq$ in suitably rescaled intervals.  
Let $I_L$ and $I_x$ denote the intervals
$$\left[-\frac{1}{L},\frac{1}{L}\right] \text{ and }\left[\frac{-s}{A^2 \pi(x)},\frac{s}{A^2 \pi(x)}\right]$$
respectively.  For sufficiently large values of $L,$
\begin{equation}\label{cancel-sines}
\begin{split}
&\frac{L}{2A\pi(x)}\sum_{p \neq q \leq x} \chi_{I_L}(\thtp - \psi) \chi_{I_L} ( \thtq - \psi)\chi_{I_x}(\thtp - \thtq)\\
&= \frac{1}{2}\cdot \frac{L}{2A\pi(x)}\sum_{p \neq q \leq x} \chi_{I_L}(\pm\thtp - \psi) \chi_{I_L} ( \pm\thtq - \psi)(\chi_{I_x}(\pm\thtp \pm \thtq).
\end{split}
\end{equation}
\begin{remark}
The additional usage of the negatives of Hecke angles enables us to reduce pair correlation sums for Hecke angles to cosine sums.  We then use the multiplicative relations among Hecke eigenvalues to evaluate these sums.  
\end{remark}
We consider smooth analogues of 
$$\frac{L}{4A\pi(x)}\left(\sum_{p \neq q \leq x} \chi_{I_L}(\pm\thtp - \psi) \chi_{I_L} ( \pm \thtq - \psi)\chi_{I_f}(\pm\thtp \pm \thtq)\right).$$
We choose real-valued, even functions $\rho,\,g \in C^{\infty}(\mathbb R),$ both of which have compactly supported Fourier transforms.  Let 
\begin{equation*}
\rho_L(\theta) := \sum_{n \in \Z} \rho(L(\theta + n))
\qquad \mathrm{and} \qquad
G_{x}(\theta) := \sum_{n \in \Z} g(\pi(x)\,(\theta + n)).
\end{equation*}
Both $\rho_L(\theta)$ and $G_{x}(\theta)$ are periodic functions with Fourier expansions
$$
\rho_L(\theta) = \frac{1}{L} \sum_{|l| \ll L} \widehat{\rho}\left(\frac{l}{L}\right) e(l \theta)
\qquad \mathrm{and} \qquad
G_{x}(\theta) = \frac{1}{\pi(x)} \sum_{|n| \ll \pi(x)} \widehat{g}\left(\frac{n}{\pi(x)}\right) e(n \theta).
$$
We define
$$R_2(g,\rho)(f) := \frac{L}{4A\pi(x)}\sum_{ p \neq q \leq x} \rho_L(\pm\thtp - \psi) \rho_L(\pm\thtq - \psi) G_{x} (\pm\thtp \pm \thtq).$$
Using the Fourier expansions of $G_{x}$ and $\rho_L,$ we have,
\begin{align}\label{R2}
R_2(g,\rho)(f) &= \frac{L}{4A\pi(x)}\frac{1}{\pi(x)L^2}\sum_{ p \neq q \leq x} \sum_l \widehat{\rho}\left(\frac{l}{L}\right) e(-l \psi) e( \pm l\thtp)\sum_{l'} \widehat{\rho}\left(\frac{l'}{L}\right) e(-l' \psi) e( \pm l' \thtq) \nonumber \\
& \qquad \sum_{n}\widehat{g}\left(\frac{n}{\pi(x)}\right) e( \pm n\thtp \pm n\thtq) \nonumber \\
&= \frac{1}{4A \pi(x)^2L} \sum_{ p \neq q \leq x} \sum_l \widehat{\rho}\left(\frac{l}{L}\right) e(-l \psi) 2 \cos 2\pi l \thtp \sum_{l'} \widehat{\rho}\left(\frac{l'}{L}\right) e(-l' \psi) 2\cos 2\pi l' \thtq\\
& \qquad \sum_{n}\widehat{g}\left(\frac{n}{\pi(x)}\right) (2\cos 2\pi n \thtp)  (2\cos 2\pi n \thtq). \nonumber
\end{align}
Since both $\rho$ and $g$ are real valued and even, the same applies to their Fourier transforms.  We deduce
\begin{align}\label{R3}
R_2(g,\rho)(f) &= \frac{1}{4A \pi(x)^2L}\sum_{ p \neq q \leq x} \left[2 \widehat{\rho}(0) + \sum_{l \geq 1} \widehat{\rho}\left(\frac{l}{L}\right) (2 \cos 2\pi l \psi) (2 \cos 2\pi l \thtp) \right] \nonumber \\
&\quad \quad \quad \quad \quad \quad \quad \quad \left[2 \widehat{\rho}(0) + \sum_{l' \geq 1} \widehat{\rho}\left(\frac{l'}{L}\right) (2 \cos 2\pi l' \psi) (2 \cos 2\pi l' \thtq)\right]\\
&\quad \quad \quad \quad \quad \quad \quad \quad \left[4 \widehat{g}(0) + \sum_{n \geq 1} 2\widehat{g}\left(\frac{n}{\pi(x)}\right) (2 \cos 2\pi n \thtp) (2 \cos 2\pi n \thtq)\right]. \nonumber
\end{align}
 We recall the following classical result (see for example, \cite[Lemma 1]{Serre}) that encodes recursive relations between $a_f(p^k),\,k \geq 1.$ 
\begin{lemma}\label{mult}
For a prime $p$ and an integer $m \geq 0,$ 
\begin{equation}\label{mult}
2 \cos 2 \pi m \thtp = 
\begin{cases}
2 &\text{ if }m = 0\\
a_f(p^{2m}) - a_f(p^{2m-2}) &\text{ if }m \geq 1.
\end{cases}
\end{equation}
\end{lemma}
We denote, for a positive integer $m,$  $$A_p(m) := 2 \cos 2 \pi m \thtp = a_f(p^{2m}) - a_f(p^{2m-2}).$$
By equation \eqref{R3}, we deduce,
\begin{align}\label{R4}
R_2(g,\rho)(f) &=  \frac{1}{4A \pi(x)^2L}\left[16 \widehat{g}(0)\widehat{\rho}(0)^2 \pi_N(x)(\pi_N(x) - 1)  + 8\widehat{g}(0)\widehat{\rho}(0)\sum_{l' \geq 1} \widehat{\rho}\left(\frac{l'}{L}\right) (2 \cos 2\pi l' \psi)  \sum_{p \neq q \leq x}A_q(l')\right. \nonumber \\
& +8\widehat{g}(0)\widehat{\rho}(0)\sum_{l \geq 1} \widehat{\rho}\left(\frac{l}{L}\right) (2 \cos 2\pi l \psi)  \sum_{p \neq q \leq x}A_p(l) + 8 \widehat{\rho}(0)^2 \sum_{n \geq 1}\widehat{g}\left(\frac{n}{\pi(x)}\right)\left( \sum_{p \neq q \leq x} A_p(n)A_q(n)\right) \nonumber \\ 
& + 4\widehat{g}(0)\sum_{l,l' \geq 1} \widehat{\rho}\left(\frac{l'}{L}\right)\widehat{\rho}\left(\frac{l}{L}\right)(2 \cos 2\pi l\psi)(2 \cos 2\pi l' \psi)  \sum_{p \neq q \leq x}A_p(l)A_q(l')\\
& + 4\widehat{\rho}(0)\sum_{l',n \geq 1} \widehat{\rho}\left(\frac{l'}{L}\right)\widehat{g}\left(\frac{n}{\pi(x)}\right)(2 \cos 2\pi l' \psi)  \sum_{p \neq q \leq x}A_q(l')A_p(n) A_q(n) \nonumber \\
& + 4\widehat{\rho}(0)\sum_{l,n \geq 1} \widehat{\rho}\left(\frac{l}{L}\right)\widehat{g}\left(\frac{n}{\pi(x)}\right)(2 \cos 2\pi l \psi)  \sum_{p \neq q \leq x}A_p(l)A_p(n)A_q(n) \nonumber \\
  & \left. + 2\sum_{l,l',n\geq 1}  \widehat{\rho}\left(\frac{l}{L}\right)\widehat{\rho}\left(\frac{l'}{L}\right)\widehat{g}\left(\frac{n}{\pi(x)}\right) (2 \cos 2\pi l \psi)(2 \cos 2\pi l' \psi) \sum_{p \neq q \leq x}A_p(l)A_q(l')A_p(n)A_q(n) \right]. \nonumber
\end{align}

\subsection{Multiplicative relations among Hecke eigenvalues}\label{Multiplicative}
In order to simplify each component of the right hand side of equation \eqref{R4}, we recall multiplicative relations between Hecke eigenvalues as expressed in the following lemma:
  \begin{lemma}\label{Hecke-multiplicative}
For primes $p_1,\,p_2 $ coprime to the level $N$ and nonnegative integers $i,\,j,$ 
$$ a_f(p_1^i)a_f(p_2^j) = 
\begin{cases}
a_f(p_1^ip_2^j) &\text{ if }p_1 \neq p_2\\
\sum_{l = 0}^{\min{(i,j)}}a_f(p_1^{i + j - 2l}) &\text{ if }p_1 = p_2.
\end{cases}$$
Moreover, if $p_1 = p_2,$ then
$$
\left(a_f(p_1^{m_1}) - a_f(p_1^{m_1-2})\right) \left(a_f(p_2^{m_2}) - a_f(p_2^{m_2-2})\right) $$
$$ = \begin{cases}
a_f(p_1^{m_1+ m_2}) - a_f(p_1^{m_1+ m_2-2}) + a_f(p_1^{|m_1- m_2|}) - a_f(p_1^{|m_1- m_2| - 2}) &\text{ if }|m_1 - m_2| \geq 2,\\
a_f(p_1^{m_1 +m_2}) + a_f(p^{|m_2-m_1|}) - a_f(p_1^{m_1 + m_2 - 2}) &\text{ if }|m_1-m_2| = 1,\\
a_f(p_1^{2m_1}) - a_f(p_1^{2m_1-2}) + 2,&\text{ if }m_1 = m_2.
\end{cases}$$
\end{lemma}
As before, we denote, for a positive integer $m,$  $$A_p(m) := 2 \cos 2 \pi m \thtp = a_f(p^{2m}) - a_f(p^{2m-2})$$
and write the following equations (equations \eqref{mult-1} to \eqref{mult-5}), which follow from Lemma \ref{Hecke-multiplicative}.
For $n,\,l,\,l' \geq 1,$ and for distinct primes $p$ and $q,$  we have
\begin{equation}\label{mult-1}
A_p(n) A_q(n) = a_f(p^{2n})a_f(q^{2n}) - a_f(p^{2n-2})a_f(q^{2n}) - a_f(p^{2n})a_f(q^{2n-2}) + a_f(p^{2n-2})a_f(q^{2n-2}),
\end{equation}
\begin{equation}\label{mult-2}
A_p(l) A_q(l') = a_f(p^{2l})a_f(q^{2l'}) - a_f(p^{2l-2})a_f(q^{2l'}) - a_f(p^{2l})a_f(q^{2l'-2}) + a_f(p^{2l-2})a_f(q^{2l'-2}),
\end{equation}
\begin{equation}\label{mult-3}
\begin{split}
&A_q(l')A_p(n) A_q(n)\\
&= \left(a_f(p^{2n}) - a_f(p^{2n-2})\right)
\left(\begin{cases}
a_f(q^{2(n+l')}) - a_f(q^{2(n+l')-2}) + a_f(q^{2|n-l'|}) - a_f(q^{2|n-l'| - 2}) &\text{ if }|n - l'| \geq 1,\\
a_f(q^{4n}) - a_f(q^{4n-2}) + 2,&\text{ if }n = l'
\end{cases}\right),
\end{split}
\end{equation}

\begin{equation}\label{mult-4}
\begin{split}
&A_p(l)A_p(n) A_q(n)\\
&= \left(a_f(q^{2n}) - a_f(q^{2n-2})\right)
\left(\begin{cases}
a_f(p^{2(n+l)}) - a_f(p^{2(n+l)-2}) + a_f(p^{2|n-l|}) - a_f(p^{2|n-l| - 2}) &\text{ if }|n - l| \geq 1,\\
a_f(p^{4n}) - a_f(p^{4n-2}) + 2,&\text{ if }n = l
\end{cases}\right)
\end{split}
\end{equation}

and 
\begin{equation}\label{mult-5}
\begin{split}
A_p(l)A_q(l')A_p(n) A_q(n) &=\left(\begin{cases}
a_f(p^{2(n+l)}) - a_f(p^{2(n+l)-2}) + a_f(p^{2|n-l|}) - a_f(p^{2|n-l| - 2}) &\text{ if }|n - l| \geq 1,\\
a_f(p^{4n}) - a_f(p^{4n-2}) + 2,&\text{ if }n = l
\end{cases}\right)\\
&\quad  \left(\begin{cases}
a_f(q^{2(n+l')}) - a_f(q^{2(n+l')-2}) + a_f(q^{2|n-l'|}) - a_f(q^{2|n-l'| - 2}) &\text{ if }|n - l'| \geq 1,\\
a_f(q^{4n}) - a_f(q^{4n-2}) + 2,&\text{ if }n = l'
\end{cases}\right).
\end{split}
\end{equation}
\bigskip

\section{Averaging and the Eichler Selberg trace formula}\label{Sato-Tate}

While the Sato-Tate conjecture remained unproven, interesting variants of it were considered.  Sarnak \cite{Sarnak} shifted perspectives and addressed a vertical variant of the Sato-Tate conjecture in the case of primitive Maass cusp forms.  For a fixed prime $p,$ he obtained a distribution measure for the $p$-th coefficients of Maass Hecke eigenforms averaged over Laplacian eigenvalues.  This perspective was developed by Serre \cite{Serre} and independently, Conrey, Duke and Farmer \cite{CDF}.  In the vertical alternative of the (holomorphic) Sato-Tate sequence, they fix a prime $p$ and varying over all newforms $f,$ consider the multisets
$$ A_p(N,k) = \{\thtp,\,f \in \mathcal F_{N,k}\} \subseteq [0,1].$$
It was shown that as $N + k \to \infty,$ with $k$ even and $(p,N) = 1,$ the multisets $A_p(N,k)$ are equidistributed in $[0,1]$ with respect to the $p$-adic Plancherel measure
$$\mu_p(t) = \frac{p+1}{(p^{1/2} + p^{-1/2})^2 - 4 \cos^2 \pi t}\mu_{\infty}(t).$$
That is, for any $[a,b] \subset [0,1],$
$$\lim_{N+k \to \infty \atop {k \text { even } \atop {(p,N) = 1}}} \frac{1}{|\mathcal F_{N,k}|}\#\{f \in \mathcal F_{N,k}:\, \thtp \in [a,b]\} = \int_a^b \mu_p(t) dt.$$
In the vertical case, one has a readily available tool to evaluate the Weyl sums $\sum_{f} e(m\thtp),$ namely the Eichler-Selberg trace formula.  Therefore, the Weyl sums and the Weyl limits are easier to evaluate.

Conrey, Duke and Farmer \cite{CDF} also consider the average version of the Sato-Tate sequence (\ref{ST-sequence}) by varying the primes as well as the eigenforms.  For ease of computation, they focus on $N=1.$  By average Sato-Tate family, we mean the family
$$A(N,k) = \{\thtp,\,p \text{ prime },\,(p,N) = 1,\,f \in \mathcal F_{N,k}\} \subseteq [0,1].$$
In this case, with extra averaging over the newforms, the trace formula tells us that the ``double" Weyl limits
$$c_m := \lim_{x \to \infty,\, k \to \infty \atop { k \text{ even }} } \frac{1}{\pi(x)|\mathcal F_{1,k}|}\sum_{p \leq x,\,f \in \mathcal F_{1,k}} e\left(m\thtp\right)$$
match the Weyl limits (\ref{limits}) of the Sato-Tate family, provided 
$$\frac{\log k}{x} \to \infty \text{ as }x \to \infty.$$
Therefore, they showed  that the families $A(1,k)$ are equidistributed with respect to $\mu_{\infty}(t),$ provided the weights $k$ grow suitably faster than the primes.  In 2006, Nagoshi \cite{N} proved the above theorem with relaxed growth conditions on $k,$ that is,
$$\frac{\log k}{\log x} \to \infty \text{ as }x \to \infty.$$ The above results can be easily generalized to any $N > 1.$  In this article, we draw upon this point of view and in a similar manner, consider the pair correlation function for the Sato-Tate sequences averaged over the Hecke eigenforms.

We now state two propositions, which follow from the well-known Eichler-Selberg trace formula  for the traces of Hecke operators acting on spaces $S(N,k).$  These propositions will play an important role in evaluating the terms of $\sum_{f \in \mathcal F_{N,k}} R_2(g,\rho)(f).$  Here, we state both these results for squarefree levels $N.$  In the next section, we will specialise to prime $N$ in order to prove Theorem \ref{Theorem2}.

\begin{prop}\label{trace}
Let $k$ be a positive even integer and $N$ be a positive squarefree integer.  For a positive integer $n>1$ such that $(n,N) = 1,$ we have
$$\sum_{f \in \mathcal F_{N,k}} a_f(n) = 
\begin{cases}
\frac{|\mathcal F_{N,k}|}{n^{1/2}} + \O\left(4^{\nu(N)} n^{c'}\right)&\text{ if }n\text{ is a square }\\
\O\left(4^{\nu(N)} n^{c'}\right)&\text{ otherwise}.
\end{cases}
$$
Here, $\nu(N)$ refers to the number of prime divisors of $N,$ $c'$ is a fixed number greater than 1 and the implied constant in the error term is absolute.
\end{prop}
\begin{proof}
The Eichler-Selberg trace formula describes the traces of the Hecke operators $T_n$ acting on spaces $S(N,k)$ for $(n,N) = 1.$  From this, one may derive the trace of $T_n$ acting on subspaces of primitive cusp forms, $S^{\text{new}}(N,k)$ contained in $S(N,k).$  The sum $\sum_{f \in \mathcal F_{N,k}} a_f(n)$ is precisely the trace of the normalized Hecke operator 
$$\frac{T_n}{n^{\frac{k-1}{2}}}$$ acting on $S^{\text{new}}(N,k).$
The Eichler-Selberg trace formula has been explicitly stated in many sources (see, for example, \cite[Section 4]{Serre}).  For a derivation of the trace of $T_n$ acting on $S^{\text{new}}(N,k)$ from the Eichler-Selberg trace formula as well as the proof of Proposition \ref{trace}, we refer the reader to \cite[Section 3]{MS2}.
\end{proof} 

From the Eichler-Selberg trace formula and related estimates in Proposition \ref{trace}, we derive the following proposition:
\begin{prop}\label{pq}
Let $a$ and $b$ be nonnegative integers and let $N$ be squarefree.  Then,
$$\left \langle \sum_{p \neq q \leq x}  a_f(p^{2a}q^{2b})\right \rangle := \sum_{p \neq q \leq x} \frac{1}{|\mathcal F_{N,k}|} \sum_f a_f(p^{2a}q^{2b})$$
$$= \begin{cases}
\pi_N(x)(\pi_N(x) - 1) &\text{ if }a = b = 0,\\
\O\left(\pi_N(x) \log \log x\right) + \O\left(\frac{\pi_N(x)^2 x^{2(a+b)c'}8^{\nu(N)}}{kN}\right) &\text{ otherwise.}
\end{cases}$$
\end{prop}
\begin{proof}
The case $a = b = 0$ needs no explanation.
Otherwise, if at least one of $a$ or $b$ is nonzero, by an application of Proposition \ref{trace}, we have, for $p \neq q,$
$$\sum_f a_f(p^{2a}q^{2b}) = |\mathcal F_{N,k}| \frac{1}{p^a q^b} + \O\left(4^{\nu(N)} p^{2ac'}q^{2bc'}\right),$$
$c'$ being an absolute constant.

Using the estimate
\begin{equation}\label{primes-reciprocals}
\sum_{p \leq x}\frac{1}{p^m} = \begin{cases}
\O(\log \log x) &\text{ if }m = 1\\
\O(1) &\text{ if }m \geq 2,
\end{cases}
\end{equation}
we have, 
\begin{equation}\label{term1}
\sum_{p \neq q \leq x} \frac{1}{p^a q^b} = \begin{cases}
\O\left(\log \log x\right)^2 &\text{ if } a,b \geq 1\\
\O\left(\pi_N(x)\log \log x \right)&\text{ if }a \text{ or }b = 0.
\end{cases}
\end{equation}
We also know (see for example \cite[Remark 11]{MS2}) that for squarefree $N,$
$$|\mathcal F_{N,k}|  = \phi(N) \left(\frac{k-1}{12}\right) + \O\left(2^{\nu(N)}\right),$$
where $\phi(N)$ denotes the Euler $\phi$-function.
By elementary estimates,
$$\frac{1}{|\mathcal F_{N,k}|}4^{\nu(N)} p^{2ac'}q^{2bc'} \ll \frac{p^{2ac'}q^{2bc'}2^{\nu(N)}}{k \phi(N)}.$$
Thus, 
\begin{equation}\label{term2}
\sum_{p \neq q \leq x} \frac{1}{|\mathcal F_{N,k}|}4^{\nu(N)} p^{2ac'}q^{2bc'} \ll \frac{\pi_N(x)^2 x^{2ac' +2bc'}8^{\nu(N)}}{k N}.
\end{equation}
The proposition follows immediately from equations \eqref{term1} and \eqref{term2}. 
\end{proof}
\bigskip

\section{Proof of Theorem 1}\label{Thm2-proof}

In this section, we put together information from Sections \ref{smooth} and \ref{Sato-Tate} to prove Theorem \ref{Theorem2}.  For a function $U:\,\mathcal F_{N,k} \to \C,$ we define
$$\langle U(f) \rangle := \frac{1}{\mathcal F_{N,k}} \sum_{f \in \mathcal F_{N,k}} U(f).$$
We start with the following proposition to evaluate the components of $\langle R_2(g,\rho)(f) \rangle$ from equation \eqref{R4}.
\begin{prop}\label{inner-1}

Let $N$ and $k$ be positive integers, with $N$ prime and $k$ even.  Let  us choose real-valued, even functions $g,\,\rho \in C^{\infty}(\R)$ with compactly supported Fourier transforms and let $L$ be a positive integer.   
There is an absolute positive constant $c$ such that:
\begin{enumerate}
\item[{\bf (a)}] 
\begin{equation*}
\begin{split}
&\frac{1}{4A \pi(x)^2L} 8\widehat{g}(0)\widehat{\rho}(0)\sum_{l' \geq 1} \widehat{\rho}\left(\frac{l'}{L}\right) (2 \cos 2\pi l' \psi) \left \langle\sum_{p \neq q \leq x} A_q(l') \right \rangle\\
&= -8 \widehat{g}(0)\widehat{\rho}(0)\widehat{\rho}\left(\frac{1}{L}\right)2 \cos 2\pi \psi \frac{\pi(x)(\pi(x) - 1)}{4A \pi(x)^2L} + \O\left(\frac{\log \log x}{\pi(x)}\right) + \O\left(\frac{x^{ L c}}{kN}\right).
\end{split}
\end{equation*}

\item[{\bf (b)}] 
\begin{equation*}
\begin{split}
&\frac{1}{4A \pi(x)^2L} 8\widehat{g}(0)\widehat{\rho}(0)\sum_{l \geq 1} \widehat{\rho}\left(\frac{l}{L}\right) (2 \cos 2\pi l \psi) \left \langle\sum_{p \neq q \leq x} A_p(l) \right \rangle\\
&= -8 \widehat{g}(0)\widehat{\rho}(0)\widehat{\rho}\left(\frac{1}{L}\right)2 \cos 2\pi \psi \frac{\pi(x)(\pi(x) - 1)}{4A \pi(x)^2L} + \O\left(\frac{\log \log x}{\pi(x)}\right) + \O\left(\frac{x^{ L c}}{kN}\right).
\end{split}
\end{equation*}
\item[{\bf (c)}] 
\begin{equation*}
\begin{split}
&\frac{1}{4A \pi(x)^2L} 8 \widehat{\rho}(0)^2 \sum_{n \geq 1}\widehat{g}\left(\frac{n}{\pi(x)}\right)\left\langle\sum_{p \neq q \leq x} A_p(n)A_q(n)\right \rangle\\
&= 8 \widehat{\rho}(0)^2\widehat{g}\left(\frac{1}{\pi(x)}\right)\frac{\pi(x)(\pi(x) - 1)}{4A \pi(x)^2L} + \O\left(\frac{\log \log x}{L}\right) + \O\left(\frac{x^{\pi(x) c}}{kN}\right).
\end{split}
\end{equation*}
\item[{\bf (d)}] 
\begin{equation*}
\begin{split}
&\frac{1}{4A \pi(x)^2L}4\widehat{g}(0)\sum_{l,l' \geq 1} \widehat{\rho}\left(\frac{l'}{L}\right)\widehat{\rho}\left(\frac{l}{L}\right)(2 \cos 2\pi l\psi)(2 \cos 2\pi l' \psi) \left \langle \sum_{p \neq q \leq x}A_p(l)A_q(l') \right \rangle\\
&= 4 \widehat{g}(0)\widehat{\rho}\left(\frac{1}{L}\right)^2 (2 \cos 2\pi \psi)^2 \frac{\pi(x)(\pi(x) - 1)}{4A \pi(x)^2L} \O\left(\frac{\log \log x}{\pi(x)}\right) + \O\left(\frac{x^{L  c}}{kN}\right).
\end{split}
\end{equation*}

\item[{\bf (e)}] 
\begin{equation*}
\begin{split}
&\frac{1}{4A \pi(x)^2L} 4\widehat{\rho}(0)\sum_{l',n \geq 1} \widehat{\rho}\left(\frac{l'}{L}\right)\widehat{g}\left(\frac{n}{\pi(x)}\right)(2 \cos 2\pi l' \psi) \left \langle \sum_{p \neq q \leq x} A_q(l')A_p(n) A_q(n) \right \rangle\\
&= \left\{4\widehat{\rho}(0)\widehat{\rho}\left(\frac{2}{L}\right)\widehat{g}\left(\frac{1}{\pi(x)}\right)(2 \cos 4\pi \psi)  - 8\widehat{\rho}(0)\widehat{\rho}\left(\frac{1}{L}\right)\widehat{g}\left(\frac{1}{\pi(x)}\right)(2 \cos 2\pi \psi)\right\}\frac{\pi(x)(\pi(x) - 1)}{4A \pi(x)^2L}\\
& + \O\left(\frac{L\log \log x}{\pi(x)}\right) + \O\left(\frac{x^{\pi(x) c}}{kN}\right).
\end{split}
\end{equation*}

\item[{\bf (f)}] 
\begin{equation*}
\begin{split}
&\frac{1}{4A \pi(x)^2L} 4\widehat{\rho}(0)\sum_{l,n \geq 1} \widehat{\rho}\left(\frac{l}{L}\right)\widehat{g}\left(\frac{n}{\pi(x)}\right)(2 \cos 2\pi l \psi) \left\langle \sum_{p \neq q \leq x} A_p(l)A_p(n) A_q(n) \right\rangle \\
&= \left\{4\widehat{\rho}(0)\widehat{\rho}\left(\frac{2}{L}\right)\widehat{g}\left(\frac{1}{\pi(x)}\right)(2 \cos 4\pi \psi)  - 8\widehat{\rho}(0)\widehat{\rho}\left(\frac{1}{L}\right)\widehat{g}\left(\frac{1}{\pi(x)}\right)(2 \cos 2\pi \psi)\right\}\frac{\pi(x)(\pi(x) - 1)}{4A \pi(x)^2L}\\
& + \O\left(\frac{L\log \log x}{\pi(x)}\right) + \O\left(\frac{x^{\pi(x) c}}{kN}\right).
\end{split}
\end{equation*}

\item[{\bf (g)}] \begin{equation*}
\begin{split}
&\frac{1}{4A \pi(x)^2L} \sum_{l,l',n\geq 1}  2\widehat{\rho}\left(\frac{l}{L}\right)\widehat{\rho}\left(\frac{l'}{L}\right)\widehat{g}\left(\frac{n}{\pi(x)}\right) (2 \cos 2\pi l \psi)(2 \cos 2\pi l' \psi)\left\langle\sum_{p \neq q \leq x}A_p(l)A_q(l')A_p(n)A_q(n)\right\rangle\\
&= \frac{\pi(x)(\pi(x) - 1)}{4A \pi(x)^2L}\sum_{n,l,l' \geq 1}2\widehat{g}\left(\frac{n}{\pi(x)}\right)\widehat{\rho}\left(\frac{l}{L}\right)\widehat{\rho}\left(\frac{l'}{L}\right)(2 \cos 2\pi l \psi)(2 \cos 2\pi l' \psi)\mathcal T(n,l,l')\\
&+ \O\left(\frac{L\log \log x}{\pi(x)}\right) + \O\left(\frac{L^2 x^{\pi(x) c}}{kN}\right),
\end{split}
\end{equation*}
where $$\mathcal T(n,l,l') = \begin{cases}
-2 &\text{ if }|n-l| = 1,\,l' = n,\\
-2&\text{ if }l = n,\,|n-l'| = 1,\\
\,1&\text{ if } |n-l| = 1 = |n - l'|\\
\,4&\text{ if } n = l = l',\\
\,0&\text{ otherwise.}
\end{cases}$$
\end{enumerate}
\end{prop}
\begin{proof}
In the calculations that follow, we assume that the support of $\widehat{\rho}$ is contained in $[-B,B],$ while the support of $\widehat {g}$ is contained in $[-D,D],$ say.

Parts  (a) - (f) have similar proofs, but we have to exercise care in identifying the leading terms of the sums concerned.  By leading terms, we mean terms containing $a_f(p^0)a_f(q^0).$  Indeed, for (a), we note that $A_q(1) = a_f(q^2) - a_f(1).$ Therefore, the leading term comes from $l' = 1$ and contributes $-a_f(1) = -1.$
That is,
\begin{equation*}
\begin{split}
&8\widehat{g}(0)\widehat{\rho}(0)\sum_{l' \geq 1} \widehat{\rho}\left(\frac{l'}{L}\right) (2 \cos 2\pi l' \psi) \sum_{p \neq q \leq x} A_q(l')\\
&= 8\widehat{g}(0)\widehat{\rho}(0)\widehat{\rho}\left(\frac{1}{L}\right) (2 \cos 2\pi \psi) \sum_{p \neq q \leq x} (-1) + \text{ remaining terms, }
\end{split}
\end{equation*}
where the remaining terms will contain $a_f(p^{2a}q^{2b})$ with $(a,b) \neq (0,0).$  Thus, applying Proposition \ref{pq} for prime $N$ and observing that the values of $l'$ run up to $\lfloor LB \rfloor,$ we have
\begin{equation*}
\begin{split}
&\frac{1}{4A \pi(x)^2L} 8\widehat{g}(0)\widehat{\rho}(0)\sum_{l' \geq 1} \widehat{\rho}\left(\frac{l'}{L}\right) (2 \cos 2\pi l' \psi) \left \langle \sum_{p \neq q \leq x} A_q(l') \right \rangle\\
&= -8 \widehat{g}(0)\widehat{\rho}(0)\widehat{\rho}\left(\frac{1}{L}\right)2 \cos 2\pi \psi \frac{\pi(x)(\pi(x) - 1)}{4A \pi(x)^2L} + O\left(\frac{1}{\pi(x)^2L}\sum_{l' =1 }^{\lfloor LB \rfloor} \left|\widehat{\rho}\left(\frac{n}{L}\right)\right|\left(\pi(x) \log \log x + \frac{\pi(x)^2x^{2l'c'}}{kN}\right)\right)\\
&=  -8 \widehat{g}(0)\widehat{\rho}(0)\widehat{\rho}\left(\frac{1}{L}\right)2 \cos 2\pi \psi \frac{\pi(x)(\pi(x) - 1)}{4A \pi(x)^2L} + \O\left(\frac{\log \log x}{\pi(x)}\right) + \O\left(\frac{x^{2 BL c'}}{kN}\right).
\end{split}
\end{equation*}
This proves (a).  Part (b) follows mutatis mutandis.

In order to prove (c), we apply equation \eqref{mult-1} and observe that for $n = 1,$
$a_f(p^{2n-2})a_f(q^{2n-2}) = 1.$  All remaining terms of $A_p(n)A_q(n)$ for the case $n=1$ as well as all the terms of $ A_p(n)A_q(n),\,n>1$ consist of combinations $a_f(p^{2a}q^{2b})$ where $(a,b) \neq (0,0).$  Moreover, the sum $ \sum_{n \geq 1}\widehat{g}\left(\frac{n}{\pi(x)}\right)\sum_{p \neq q \leq x}A_p(n)A_q(n)$ runs up to $n = \lfloor D\pi(x) \rfloor.$  Thus, on the left hand side of (c), we separate the leading term,
$$\frac{1}{4A \pi(x)^2L} 8 \widehat{\rho}(0)^2 \widehat{g}\left(\frac{1}{\pi(x)}\right)\sum_{p \neq q \leq x} a_f(p^{2.1-2})a_f(q^{2.1-2})= \frac{1}{4A \pi(x)^2L} 8 \widehat{\rho}(0)^2 \widehat{g}\left(\frac{1}{\pi(x)}\right)\sum_{p \neq q \leq x} 1 $$
from the rest of the sum and get
\begin{equation*}
\begin{split}
&\frac{1}{4A \pi(x)^2L} 8 \widehat{\rho}(0)^2 \sum_{n \geq 1}\widehat{g}\left(\frac{n}{\pi(x)}\right)\sum_{p \neq q \leq x} A_p(n)A_q(n)\\
&= \frac{1}{4A \pi(x)^2L} 8 \widehat{\rho}(0)^2 \widehat{g}\left(\frac{1}{\pi(x)}\right)\left(\sum_{p \neq q \leq x} 1\right) + \text{ remaining terms. }
\end{split}
\end{equation*}
Here, the remaining terms will contain $a_f(p^{2a}q^{2b})$ with either $a>0$ or $b>0.$  We apply Proposition \ref{pq}.  Thus, we have
\begin{equation*}
\begin{split}
&\frac{1}{4A \pi(x)^2L} 8 \widehat{\rho}(0)^2 \sum_{n \geq 1}\widehat{g}\left(\frac{n}{\pi(x)}\right)\left\langle\sum_{p \neq q \leq x}A_p(n)A_q(n)\right \rangle\\
&= \frac{1}{4A \pi(x)^2L} 8 \widehat{\rho}(0)^2 \widehat{g}\left(\frac{1}{\pi(x)}\right)\left(\sum_{p \neq q \leq x} 1\right) + O\left(\frac{1}{\pi(x)^2L}\sum_{n =1 }^{\lfloor D\pi(x) \rfloor} \left|\widehat{g}\left(\frac{n}{\pi(x)}\right)\right|\left(\pi(x) \log \log x + \frac{\pi(x)^2x^{4nc'}}{kN}\right)\right)\\\\
&= \frac{1}{4A \pi(x)^2L} 8 \widehat{\rho}(0)^2 \widehat{g}\left(\frac{1}{\pi(x)}\right)\left(\sum_{p \neq q \leq x} 1\right) + \O\left(\frac{\log \log x}{L}\right) + \O\left(\frac{x^{4 D\pi(x) c'}}{kN}\right).
\end{split}
\end{equation*}
This proves (c).

Similarly, for part (d), we apply equation \eqref{mult-2} and observe that for $l = l' = 1,$
$a_f(p^{2l-2})a_f(q^{2l'-2}) = 1.$  All other terms of $A_p(l)A_q(l')$ for the case $l = l'=1$ as well as all the terms of $ A_p(l)A_q(l')$ with $l$ or $l' >1$ consist of combinations $a_f(p^{2a}q^{2b})$ where $(a,b) \neq (0,0).$  Thus, by Proposition \ref{pq}, we have
\begin{equation*}
\begin{split}
&\frac{1}{4A \pi(x)^2L}4\widehat{g}(0)\sum_{1 \leq l,l' \ll L} \widehat{\rho}\left(\frac{l'}{L}\right)\widehat{\rho}\left(\frac{l}{L}\right)(2 \cos 2\pi l\psi)(2 \cos 2\pi l' \psi) \left \langle \sum_{p \neq q \leq x}A_p(l)A_q(l') \right \rangle\\
&= 4 \widehat{g}(0)^2\widehat{\rho}\left(\frac{1}{L}\right)^2 (2 \cos 2\pi \psi)^2 \frac{\pi(x)(\pi(x) - 1)}{4A \pi(x)^2L} + \O\left(\frac{\log \log x}{\pi(x)}\right) + \O\left(\frac{x^{4 B L  c'}}{kN}\right).
\end{split}
\end{equation*}

To prove part (e), we identify leading terms from equation \eqref{mult-3}.  We see that $A_q(l')A_p(n)A_q(n)$ contributes the term $-2a_f(p^0q^0) = -2$ when $n=l'=1$ and $a_f(p^0q^0) = 1$ when $n=1$ and $l'=2.$  
Thus,
\begin{equation*}
\begin{split}
&4\widehat{\rho}(0)\sum_{l',n \geq 1} \widehat{\rho}\left(\frac{l'}{L}\right)\widehat{g}\left(\frac{n}{\pi(x)}\right)(2 \cos 2\pi l' \psi) \sum_{p \neq q \leq x} A_q(l')A_p(n) A_q(n)\\
&= \left\{4\widehat{\rho}(0)\widehat{\rho}\left(\frac{2}{L}\right)\widehat{g}\left(\frac{1}{\pi(x)}\right)(2 \cos 4\pi \psi)  - 8\widehat{\rho}(0)\widehat{\rho}\left(\frac{1}{L}\right)\widehat{g}\left(\frac{1}{\pi(x)}\right)(2 \cos 2\pi \psi)\right\}\left(\sum_{p \neq q \leq x} 1 \right)\\
&+ \text{ remaining terms, }
\end{split}
\end{equation*}
where the remaining terms will contain $a_f(p^{2a}q^{2b})$ with $(a,b) \neq (0,0).$  The values of $n$ in this sum run up to $\lfloor D\pi(x) \rfloor$ and $l'$ up to $\lfloor BL \rfloor.$  Just as in the previous cases, part (e) now follows by Proposition \ref{pq}.
Part (f) can be derived almost exactly as part (e), using equation \eqref{mult-4} in place of \eqref{mult-3}.

To prove part (g), the leading terms for the left hand side come from various values of $n,\,l$ and $l'.$  In fact, by equation \eqref{mult-5},
\begin{equation*}
A_p(l)A_q(l')A_p(n)A_q(n) = \left(\begin{cases}
-2 &\text{ if }|n-l| = 1,\,l' = n,\\
-2&\text{ if }l = n,\,|n-l'| = 1,\\
\,1&\text{ if } |n-l| = 1 = |n - l'|\\
\,4&\text{ if } n = l = l',\\
\end{cases}\right) + \text{ remaining terms },
\end{equation*}
where the remaining terms will contain $a_f(p^{2a}q^{2b})$ with $(a,b) \neq (0,0).$
Part (g) now follows from Proposition \ref{pq}. 
\end{proof}

We now compile information from Proposition \ref{inner-1} to prove the following theorem, which can be viewed as a smooth analogue of Theorem \ref{Theorem2}. 

\begin{thm}\label{expected}
Let us consider families $\mathcal F_{N,k}$ with prime levels $N = N(x)$ and even weights $k = k(x).$   Let  $g,\,\rho$ be real valued, even functions $\in C^{\infty}(\R)$ with compactly supported Fourier transforms.
Then, for $0 < \psi < 1$ and $L \geq 1,$ we have
\begin{equation}\label{all-prop}
\frac{1}{|\mathcal F_{N,k}|} \sum_{f \in \mathcal F_{N,k}}  R_2(g,\rho)(f) = C_{\psi} \widehat{g}(0)(\rho \ast \rho)(0) + \O\left(\frac{L\log \log x}{\pi(x)}\right) + \O\left(\frac{L^2 x^{\pi(x) c}}{kN}\right).
\end{equation}
Here $C_{\psi} = 4(2\sin^2\pi \psi)$ if $\psi = \frac{1}{2}$ and $C_{\psi} = 2(2\sin^2\pi \psi)$ if $\psi \neq \frac{1}{2}.$
In particular, if we choose $L(x) \approx \log \log x$ and $\frac{\log kN}{x} \to \infty$ as $x \to \infty,$ then
$$\lim_{x \to \infty}\frac{1}{|\mathcal F_{N,k}|} \sum_{f \in \mathcal F_{N,k}}R_2(g,\rho)(f) = C_{\psi} \widehat{g}(0)(\rho \ast \rho)(0).$$
\end{thm} 

\begin{remark}\label{rmk-thm}
In the above theorem, we may choose $\rho$ such that $\rho^2$ is of mass 1.  Furthermore, for $0 < \psi < 1,$ we may modify $\mathcal I_L$ into an interval around $\psi$ of length $2C_{\psi}/ L,$ for example, $$\left[ \psi - \frac{C_{\psi}}{ L}, \psi + \frac{C_{\psi}}{L}\right].$$ This would give us that $$\lim_{x \to \infty}\frac{1}{|\mathcal F_{N,k}|} \sum_{f \in \mathcal F_{N,k}}R_2(g,\rho)(f) = \widehat{g}(0).$$ Thus, Theorem \ref{expected} implies Theorem \ref{Theorem2}.
\end{remark}

\begin{remark}
We also note that our choice of $L(x)$ in Theorem \ref{expected} is enough to derive the asymptotics of the (average) local pair correlation function $R_{f,\psi}(s),$ but is by no means optimal.  The nature of the error terms in equation \eqref{all-prop} enables us to choose $L(x)$ of much smaller orders for the asymptotics to hold, as long as $L(x) \to \infty.$  \end{remark}

{\bf Proof of Theorem \ref{expected}.}  By equations \eqref{R3} and \eqref{R4}, we have, for $L \geq 1,$
\begin{align*}\label{crazy}
  \langle R_2 &(g,\rho) (f) \rangle = \frac{1}{4A \pi(x)^2L} \biggl \langle\sum_{ p \neq q \leq x} \biggl[2 \widehat{\rho}(0) + \sum_{l \geq 1} \widehat{\rho}\left(\frac{l}{L}\right) (2 \cos 2\pi l \psi) (2 \cos 2\pi l \thtp) \biggr]  \biggl[2 \widehat{\rho}(0) +  \\
&  \sum_{l' \geq 1} \widehat{\rho}\left(\frac{l'}{L}\right) (2 \cos 2\pi l' \psi) (2 \cos 2\pi l' \thtq) \biggr] \biggl[4 \widehat{g}(0) + \sum_{n \geq 1} 2\widehat{g}\left(\frac{n}{\pi(x)}\right) (2 \cos 2\pi n \thtp) (2 \cos 2\pi n \thtq)\biggr] \biggr \rangle \nonumber \\
&= \frac{1}{4A \pi(x)^2L} \biggl[\sum_{p \neq q \leq x} 16 \widehat{g}(0)\widehat{\rho}(0)^2  +8\widehat{g}(0)\widehat{\rho}(0)\sum_{l' \geq 1} \widehat{\rho}\left(\frac{l'}{L}\right) (2 \cos 2\pi l' \psi)  \biggl\langle \sum_{p \neq q \leq x} A_q(l') \biggr\rangle \\
&+ 8\widehat{g}(0)\widehat{\rho}(0)\sum_{l \geq 1} \widehat{\rho}\left(\frac{l}{L}\right) (2 \cos 2\pi l \psi)  \biggl \langle \sum_{p \neq q \leq x} A_p(l) \biggr \rangle + 8 \widehat{\rho}(0)^2 \sum_{n \geq 1}\widehat{g}\left(\frac{n}{\pi(x)}\right) \biggl \langle \sum_{p \neq q \leq x} A_p(n)A_q(n) \biggr \rangle  \\
&+4\widehat{g}(0)\sum_{l,l' \geq 1} \widehat{\rho}\left(\frac{l'}{L}\right)\widehat{\rho}\left(\frac{l}{L}\right)(2 \cos 2\pi l\psi)(2 \cos 2\pi l' \psi) \biggl \langle \sum_{p \neq q \leq x}A_p(l)A_q(l') \biggr \rangle  \\
& +  4\widehat{\rho}(0)\sum_{l',n \geq 1} \widehat{\rho}\left(\frac{l'}{L}\right)\widehat{g}\left(\frac{n}{\pi(x)}\right)(2 \cos 2\pi l' \psi) \biggl \langle \sum_{p \neq q \leq x} A_q(l')A_p(n) A_q(n) \biggr \rangle  \\
& + 4\widehat{\rho}(0)\sum_{l,n \geq 1} \widehat{\rho}\left(\frac{l}{L}\right)\widehat{g}\left(\frac{n}{\pi(x)}\right)(2 \cos 2\pi l \psi) \biggl \langle \sum_{p \neq q \leq x} A_p(l)A_p(n)A_q(n) \biggr \rangle  \\
&+   2\sum_{l,l',n\geq 1}  \widehat{\rho}\left(\frac{l}{L}\right)\widehat{\rho}\left(\frac{l'}{L}\right)\widehat{g}\left(\frac{n}{\pi(x)}\right) (2 \cos 2\pi l \psi)(2 \cos 2\pi l' \psi) \biggl \langle \sum_{p \neq q \leq x} A_p(l)A_q(l')A_p(n)A_q(n) \biggr \rangle \biggr]. \nonumber
\end{align*}
We apply Proposition \ref{inner-1}.  Separating the leading terms from the others, the above equals
\begin{equation}\label{all}
\begin{split}
&\frac{\pi(x)(\pi(x) - 1)}{2A \pi(x)^2L} \left(S(g,\rho) + T(g,\rho)\right) + \O\left(\frac{L\log \log x}{\pi(x)}\right) + \O\left(\frac{L^2 x^{\pi(x) c}}{kN}\right),
\end{split}
\end{equation}
where
$S(g,\rho)$ is the contribution to $\langle R(g,\rho)(f) \rangle$ from the leading terms appearing in Proposition \ref{inner-1} (a) - (f), given by
\begin{equation}\label{Sg-sum}
\begin{split}
&S(g,\rho) = 8\widehat{g}(0)\widehat{\rho}(0)^2 - 8\widehat{g}(0)\widehat{\rho}(0)\widehat{\rho}\left(\frac{1}{L}\right)(2 \cos 2 \pi \psi) +  4\widehat{g}\left(\frac{1}{\pi(x)}\right)\widehat{\rho}(0)^2 \\
&+ 2 \widehat{g}(0)\widehat{\rho}\left(\frac{1}{L}\right)^2 (2 \cos 2 \pi \psi)^2+  4 \widehat{g}\left(\frac{1}{\pi(x)}\right)\widehat{\rho}(0)\widehat{\rho}\left(\frac{2}{L}\right) 2 \cos 4 \pi \psi\\
&-8\widehat{\rho}(0)\widehat{\rho}\left(\frac{1}{L}\right)\widehat{g}\left(\frac{1}{\pi(x)}\right)\\
\end{split}
\end{equation}
and $T(g,\rho)$ is the contribution to $\langle R(g,\rho)(f) \rangle$ from the leading terms appearing in Proposition \ref{inner-1} (g), given by
\begin{equation}\label{Tg-sum}
\begin{split}
&T(g,\rho) = \sum_{n=2}^{\lfloor LB \rfloor + 1} \widehat{g}\left(\frac{n}{\pi(x)}\right)\widehat{\rho}\left(\frac{n-1}{L}\right)^2(2 \cos 2\pi (n-1)\psi)^2 \\
 &+ \sum_{n=1}^{\lfloor LB \rfloor - 1} \widehat{g}\left(\frac{n}{\pi(x)}\right)\widehat{\rho}\left(\frac{n+1}{L}\right)^2(2 \cos 2\pi (n+ 1)\psi)^2\\
&+2\sum_{n=2}^{\lfloor LB \rfloor - 1} \widehat{g}\left(\frac{n}{\pi(x)}\right)\widehat{\rho}\left(\frac{n-1}{L}\right)\widehat{\rho}\left(\frac{n+1}{L}\right)(2 \cos 2\pi (n-1)\psi)(2 \cos 2\pi (n+1)\psi)\\
&+ 4\sum_{n=1}^{\lfloor LB \rfloor} \widehat{g}\left(\frac{n}{\pi(x)}\right)\widehat{\rho}\left(\frac{n}{L}\right)^2(2 \cos 2\pi n\psi)^2 \\
&-4\sum_{n=1}^{\lfloor LB \rfloor - 1} \widehat{g}\left(\frac{n}{\pi(x)}\right)\widehat{\rho}\left(\frac{n}{L}\right)\widehat{\rho}\left(\frac{n+1}{L}\right)(2 \cos 2\pi n \psi)(2 \cos 2\pi (n+1)\psi)\\
&-4\sum_{n=2}^{\lfloor LB \rfloor } \widehat{g}\left(\frac{n}{\pi(x)}\right)\widehat{\rho}\left(\frac{n}{L}\right)\widehat{\rho}\left(\frac{n-1}{L}\right)(2 \cos 2\pi n \psi)(2 \cos 2\pi (n-1)\psi)\\
\end{split}
\end{equation}
As long as $L = L(x) \to \infty,$ we have
$$\frac{\pi(x)(\pi(x) - 1)}{2A \pi(x)^2L} S(g,\rho) \ll \frac{1}{L} \to 0 \text{ as }x \to \infty.$$
 
Let us consider the error terms in equation \eqref{all},
$$ \O\left(\frac{L\log \log x}{\pi(x)}\right) + \O\left(\frac{L^2 x^{\pi(x) c}}{kN}\right).$$
Let us choose $L(x) \approx \log \log x.$  The first error term goes to 0 as $x \to \infty.$  Suppose we choose $N = N(x)$ and $k = k(x)$ such that $\frac{\log kN}{x} \to \infty.$  Then
$$\frac{L^2 x^{\pi(x) c}}{kN}\ \approx \frac{(\log \log x)^2 x^{\pi(x) c}}{kN} \to 0 \text{ as }x \to \infty.$$
Thus, to prove the theorem, we have to show that
$$\lim_{x \to \infty} \frac{\pi(x)(\pi(x) - 1)}{2A \pi(x)^2L} T(g,\rho) = C_{\psi}\,\widehat{g}(0) (\rho \ast \rho) (0).$$
In equation \eqref{Tg-sum}, the range of each of the sums is up to $BL,$ which is $\o(\pi(x)).$  Moreover, each $\rho$ is Lipschitz continuous.  Thus,
\begin{equation*}
\begin{split}
&\frac{\pi(x)(\pi(x) - 1)}{2A \pi(x)^2L} T(g,\rho)\\
&\sim \frac{4\widehat{g}(0)}{2A} \frac{1}{L}\sum_{n=1}^{\lfloor LB \rfloor} \widehat{\rho}\left(\frac{n}{L}\right)^2 \left[ \cos^2 (2\pi (n-1)\psi) + \cos^2 (2\pi (n+1)\psi) + 2\cos (2\pi (n-1)\psi)\cos (2\pi (n+1)\psi)\right.\\
&\left. \quad\quad+4\cos^2( 2\pi n\psi) -4 \cos( 2\pi n\psi)\cos( 2\pi (n+1)\psi) - 4 \cos( 2\pi n\psi)\cos( 2\pi (n-1)\psi)\right]\\
&= \frac{64\widehat{g}(0)\sin^4(\pi \psi)}{2A} \frac{1}{L}\sum_{n=1}^{\lfloor LB \rfloor} \widehat{\rho}\left(\frac{n}{L}\right)^2  \cos^2 (2 \pi n \psi)\\
&=8A\widehat{g}(0) \frac{1}{L}\sum_{n=1}^{\lfloor LB \rfloor} \widehat{\rho}\left(\frac{n}{L}\right)^2  \cos^2 (2 \pi n \psi)\\
&= \frac{4A\widehat{g}(0)}{L}\sum_{n=1}^{\lfloor LB \rfloor}\widehat{\rho}\left(\frac{n}{L}\right)^2(1 + \cos 4 \pi n \psi).
\end{split}
\end{equation*}
We now observe that
\begin{equation*}
\begin{split}
&\lim_{L \to \infty}\frac{1}{L}\sum_{n=1}^{\lfloor LB \rfloor}\widehat{\rho}\left(\frac{n}{L}\right)^2 \cos 4 \pi n \theta\\
&= \lim_{L \to \infty} \frac{1}{2L} \sum_{n \in \Z} \widehat{\rho}\left(\frac{n}{L}\right)^2 e( 2\pi n \theta) - \lim_{L \to \infty}\frac{\rho(0)^2}{L} \\
& = \lim_{L \to \infty} \frac{1}{2L}\sum_{n \in \Z} \widehat{\rho}\left(\frac{n}{L}\right)^2 e( 2\pi n \theta)\\
&= \frac{1}{2}\lim_{L \to \infty} \sum_{n \in \Z} (\rho \ast \rho)(L(n+ 2\theta)).
\end{split}
\end{equation*}
 $(\rho \ast \rho)(L(n+ 2\theta))$ is the Fourier transform of the function
 $\widehat{\rho \ast \rho}$
 at $-L(n+ 2\theta).$  Since the function $\widehat{\rho \ast \rho} \in C_c^{\infty}(\R),$  by the Riemann-Lebesgue lemma,
 $$(\rho \ast \rho)(u) \to 0 \text{ as }|u| \to \infty.$$
 Thus, for $0 \leq \theta <1,$
 $$\lim_{L \to \infty}  (\rho \ast \rho)(L(n+ 2\theta)) = \begin{cases}
 (\rho \ast \rho)(0)&\text{ if }\theta = 0 \text{ and }n = 0\\
 0 &\text{ if }\theta = 0 \text{ and }n \neq 0\\
 0 &\text{ if }\theta = \frac{1}{2} \text{ and }n \neq -1\\
 (\rho \ast \rho)(0)&\text{ if }\theta = \frac{1}{2} \text{ and }n = -1\\
 0 &\text{ if }\theta \neq 0,\, \frac{1}{2}.
 \end{cases}
 $$
Thus, for $0 < \psi < 1,$
$$\frac{1}{L}\sum_{n=1}^{\lfloor LB \rfloor}\widehat{\rho}\left(\frac{n}{L}\right)^2(1 + \cos 4 \pi n \psi) = \begin{cases} 
(\rho \ast \rho)(0) &\text{ if } \psi = \frac{1}{2}\\
\frac{1}{2}(\rho \ast \rho)(0)&\text{ if } \psi \neq \frac{1}{2}.
\end{cases}
$$

This proves Theorem \ref{expected} and therefore, Theorem \ref{Theorem2}.
\bigskip

\section{Hilbert modular case}\label{Hilbert}

In this section, we generalize Theorem \ref{Theorem2} in the context of Hilbert modular forms.  We refer to \cite[\S 1, 2]{Shi} for an overview of definitions and basic properties of Hilbert modular forms.  Our main tool here as in the elliptic modular case is a trace formula, namely Arthur's trace formula. We will use estimates for this trace formula derived in \cite{LLW}.

Let $F$ be a totally real field of degree $d$ over $\Q$ and let $\mathcal O_F$ denote the ring of integers of $F$. For each place $v$ of $F$, let $F_v$ denote the completion of $F$ at $v$ and let $\mathcal O_{F,v}$ denote the valuation ring at $v$. Let $\mathbb A_\Q$ denote the adeles over $\Q$ and let $\mathbb A_{\Q,f}$ denote the finite adeles.   Let $G$ denote the algebraic group that is the restriction of scalars of ${\mathrm{GL}_2}_{/F}$ from $F$ to $\Q$.   Let $G_f = G(\mathbb A_{\Q,f})$  and $G_\infty = G(\R)$.  For an integral ideal $\mathfrak n \subset \mathcal O_F$, let $K_0 (\mathfrak n) \subset G_f $ be the congruence subgroup of level $\mathfrak n$. Let $\mathfrak n = \mathfrak q_1^{a_1} \cdots \mathfrak q_r^{a_r}$ be the prime factorization of $\mathfrak n.$  This congruence subgroup is defined as 
$$
K_0 (\mathfrak n) = \left\{ g = \bmatrix a & b \\ c & d \endbmatrix \in \prod_v \mathrm{GL}_2 (\mathcal O_{F,v}) \bigg| c_{\mathfrak q_i} \equiv 0\ \mod\ \mathfrak q_i^{a_i}, \forall i = 1, \dots, r  \right\}.
$$
These are adelic versions of the classical congruence subgroups $\Gamma_0 (N)$.

A weight in the Hilbert modular case is a $d$-tuple $\underline{k} = (k_1, \dots, k_d)$ of even integers. We further assume that  $k_i \ge 4$.  Let $\pi$ be an automorphic representation with respect to $G.$  We write $\pi = \pi_f \otimes \pi_\infty$ for representations $\pi_f$ of $G_f$ and $\pi_\infty$ of $G_\infty$.   Let $\Pi_{\uk} (\mathfrak n)$  denote the set of cuspidal unitary automorphic representations $\pi$ with respect to $G$ such that $\pi_f$ has a $K_0 (\mathfrak n)$-fixed vector and $\pi_\infty = \otimes_{i=1}^d D_{k_i - 1}$. Here $D_k$ is the discrete series representation of $\mathrm{GL}_2 (\R)$ with minimal $K$-type of weight $k+1$.  The set $\Pi_{\uk} (\mathfrak n)$ is a finite set. 

For $\pi$ as above, let $\mathfrak p$ be a prime ideal of $F$ for which $\pi_\gp$ is unramified. Note that for any $\pi$ with a $K_0 (\mathfrak n)$ fixed vector, the representations $\pi_\gp$ are unramified for $\gp \nmid \mathfrak n$. The Satake parameter associated to $\pi_\gp$ is a conjugacy class
$$
t_{\pi} (\gp) = \begin{pmatrix} e^{i \pi \theta_{\pi} (\gp)} & \\ & e^{-i \pi \theta_{\pi} (\gp)} \end{pmatrix} \in \mathrm{SU} (2)/\sim,
$$
where $\theta_{\pi} (\gp) \in [0, 1]$. 

The assertion that $t_\pi (\gp) \in \mathrm{SU}(2)$ is the Ramanujan conjecture for Hilbert modular forms, proved by Blasius \cite{Blasius}. For $F=\Q$ and a classical eigenform $f$, let $\pi (f)$ denote the automorphic representation associated to  $f$. The angles $\theta_{\pi (f)} (p)$ considered here are exactly those considered in earlier sections. In this section, we investigate pair correlation of $\theta_\pi (\gp)$ over families similar to the classical case.

Now assume $\pi$ is new of level $K_0 (\mathfrak n)$; that is, $\mathrm{dim}\,\pi_f^{K_0 (\mathfrak n)} = 1$. By the work of Barnet-Lamb, Gee and Geraghty~\cite{BGG}, we know that the sequence $\{\theta_\pi (\gp) \mid  \gp \nmid \mathfrak n\} \subset [0,1] $ is equidistributed with respect to the measure $2 \sin^2 (\pi \theta) d\theta$. This is the Sato-Tate equidistribution theorem for Hilbert modular forms.

Let $0 < \psi < 1.$  Let $\mathcal I_L$ denote the interval
$$
\left[\psi - \frac{1}{L},\psi + \frac{1}{L}\right].
$$  
For positive real numbers $s$ and $x,$ we define
$$
\Lpi := \Lpi\,(x,L,\psi) := \#\left\{\gp \leq x:\,(\gp,\mathfrak n) = 1,\,\theta_\pi (\gp) \in\mathcal I_L \right\},
$$
and
\begin{equation*}
R_{\pi,x,L}(s) := \frac{1}{\Lpi} \#\left \{ 1 \leq \gp \neq \gq \leq x: \begin{array}{c} \, (\gp,\mathfrak n) = (\gq, \mathfrak n) = 1,
 \theta_\pi (\gp),\,\theta_\pi (\gq) \in \mathcal I_L,\\ |H(\theta_\pi (\gp)) - H(\theta_\pi (\gq))|  \leq \frac{2s}{L\Lpi} \end{array} \right\}.
\end{equation*}
Here we use the notation $\gp \leq x$ to indicate that $\mathrm{Nm} (\gp) \le x.$ 

Choosing $L = L(x)$ to be an increasing function such that $L(x) \to \infty$ as $x \to \infty,$ we define the { local pair correlation function around $\psi$} as
$$
R_{\pi,\psi}(s) := \lim_{x \to \infty} R_{\pi,x,L}(s).
$$

By the Sato-Tate equidistribution theorem,  
\begin{equation*}
\Lpi \sim \pi_F  (x)A\frac{2}{L}  \text{ as }x \to \infty,
\end{equation*}
where $\pi_{F} (x) = \# \{\gp \mid \,\mathrm{Nm} (\gp) \le x \}$.
We may choose $L = L(x) \to \infty$ as $x \to \infty.$  Thus, as in the classical case, for $x \to \infty,$
\begin{equation*}
R_{\pi,x,L}(s) \sim \frac{L}{2A\pi_F (x)}\#\left\{ \gp \neq \gq \leq x,\,\theta_\pi (\gp),\,\theta_\pi (\gq) \in\mathcal I_L,\,\theta_\pi (\gp) - \theta_\pi (\gq) \in \left[\frac{-s}{A^2 \pi_F (x)},\frac{s}{A^2 \pi_F (x)}\right]\right\}.
\end{equation*}

We now state our main theorem in the Hilbert modular case
\begin{thm}\label{hilbert-Theorem2}
Let $0 < \psi < 1$ and $L = L(x) \approx \log \log x.$  For a fixed square free level $\mathfrak n$, we consider families $\Pi_{\uk} (\mathfrak n)$ with even weights $k = k(x)$ such that 
$$
\frac{\sum_i \log k_i}{x} \to \infty,
$$
as $x \to \infty$.
Then,
$$\lim_{x \to \infty}\frac{1}{|\Pi_{\uk} (\mathfrak n)|}\sum_{\pi \in \Pi_{\uk} (\mathfrak n)} R_{\pi,x,L}(s) = 2s.$$
That is, 
$$
 \frac{1}{|\Pi_{\uk} (\mathfrak n)|}\sum_{\pi \in \Pi_{\uk} (\mathfrak n)}R_{\pi,\psi}(s) \sim 2s
$$
under the above mentioned growth conditions on $k.$
\end{thm}

Note that in the classical case, our sequence consisted of eigenangles coming from newforms and we used estimates of Eichler-Selberg trace formula on the space of newforms. However, in the Hilbert modular case, such estimates are not available at the moment. Hence, we fix a squarefree level $\mathfrak n$ and use estimates for Arthur's trace formula for $\Pi_{\uk} (\mathfrak n)$, which we state below. 

For $m \ge 0$, let $X_m$ be polynomials such that $X_m (2 \cos  \pi \theta) = \frac{\sin (m+1) \pi \theta}{\sin \pi \theta}$. Then $2 \cos \pi m \theta = X_m (2 \cos  \pi \theta) - X_{m-2} (2 \cos 2\pi \theta)$, for $m \ge 2$.  For $m \ge 1$, we have
$$
\sum_{\pi \in \Pi_{\uk} (\mathfrak n)} \sum_{\gp \leq x} 2 \cos (2m (\pi \theta_\pi (\gp))) = \sum_{\gp \leq x} \sum_{\pi \in \Pi_{\uk}  (\mathfrak n)}  X_{2m} (2 \cos \pi \theta_\pi (\gp)) - X_{2m-2} (2 \cos \pi \theta_\pi (\gp)).
$$

We  state below the version of Arthur's trace formula that we need.
\begin{prop}[{\cite[Theorem 6.3]{LLW}}] \label{trace-Arthur}
Let $\gp_1, \dots, \gp_h$ be distinct primes coprime to a squarefree $\mathfrak n$ and let $\underline{m} = (m_1, \dots , m_h)$ be a tuple of non-negative integers. Let $\mathfrak a = \gp_1^{m_1} \cdots \gp_h^{m_h}$. Then 
$$
\sum_{\pi \in \Pi_{\uk}  (\mathfrak n)} \prod_{i=1}^h X_{m_i} (2 \cos \pi \theta_\pi (\gp_i)) =  C \mathrm{Nm} (\mathfrak n)\delta_{2|\underline{m}} \prod_{i=1}^d \frac{k_i -1}{4 \pi} \mathrm{Nm} (\mathfrak a)^{-1/2} + O(\mathrm{Nm} (\mathfrak n)^{\epsilon}\mathrm{Nm} (\mathfrak a)^{3/2});
$$
where 
\begin{itemize}
\item $C$ is a constant depending only on the number field $F$, 
\item $\delta_{2|\underline{m}}$ is one if all the $m_i$ are even and zero otherwise and
\end{itemize}
 Note also that as a consequence, we get $\# \Pi_{\uk} (\mathfrak n) = C  \mathrm{Nm} (\mathfrak n)\prod_{i=1}^d \frac{k_i -1}{4 \pi} + O(\mathrm{Nm} (\mathfrak n)^\epsilon)$. 
\end{prop}

Theorem \ref{hilbert-Theorem2} follows from Proposition \ref{trace-Arthur} using the same techniques as in earlier sections. 
\bigskip

\section{Modular forms on the hyperbolic $3$-space}\label{Bianchi}

Let $E$ be a non-CM elliptic curve over a CM field. The Sate-Tate conjecture is true for $E$, see \cite{10author}. More generally, the Sato-Tate conjecture is expected to be true for automorphic representations on $\mathrm{GL}_2$ over CM fields.  This motivates us to study pair correlation statistics for automorphic forms arising in this situation.

In this section, we prove pair correlation statistics for automorphic forms with respect to $\mathrm{SL}_{2}(\mathcal O_K)$, where $K$ is an imaginary quadratic field with class number one. Our main references in this section will be \cite{IR} and \cite{Raulf}.

Let $\{1,i,j,k\}$ denote the basis of Hamiltonian quaternions.  Viewing the upper half space $\H^3$ as a subset of Hamiltonian quaternions with vanishing fourth coordinate, we consider a point $P = z + rj \in \H^3$ where $r > 0$ and $z = x + iy \in \C.$  We equip $\H^3$ with the hyperbolic metric and consider the Laplace-Beltrami operator $\Delta$ corresponding to this metric defined as follows:
$$\Delta := r^2\left(\frac{\partial ^2}{\partial x^2} + \frac{\partial ^2}{\partial y^2} + \frac{\partial ^2}{\partial z^2}\right) - r\frac{\partial}{\partial r}.$$
Let $\mathcal O_K$ denote the ring of integers of $K$ and let $\Gamma = \mathrm{PSL}_2 (\mathcal{O}_K)$. We consider the usual action of $\Gamma$ on the hyperbolic $3$-space $\H^3$ and let $\V$ denote the volume of the quotient $\Gamma \backslash \H^3.$  

For a prime $\gp = (\varpi)$, let 
$$
\mathcal M_\gp = \{ M \in \mathrm{GL}_2(\mathcal O_K):\,\det M = \varpi\}.
$$
$\Gamma$ acts on $\mathcal M_\gp$ by left multiplication.  Let $\mathcal V_\gp$ be a system of representatives for the orbits of $\mathcal M_\gp$ modulo $\Gamma.$
For $f \in \Gamma \backslash \H^3 \to \C,$ the normalized Hecke operator $T_\gp$ is defined as
$$(T_\gp f)(P) = \frac{1}{\sqrt{\mathrm{Nm}(\gp)}}\sum_{M \in \mathcal V_\gp}f(MP),$$
where $\mathrm{Nm}(\gp)$ denotes the norm of $\gp.$ 
The Hecke operators $T_\gp$ satisfy multiplicative and recursive relations similar to those satisfied by the classical Hecke operators. That is, 
$$
 T (\gp) T (\gp^m) = T(\gp^{m+1}) + T (\gp^{m-1})
 $$
for a positive integer $m$. 

Let $\{e_j\}_{j \geq 0}$ denote an orthonormal basis of eigenfunctions with eigenvalues $\lambda_j$ for the operator $-\Delta$ in the discrete spectrum of $L^2 (\Gamma \backslash \mathcal H)$ such that for every $j \geq 0$ and for every prime $\gp \in \mathcal{O}_K,$
$$T (\gp^{m}) e_j = \rho_j (\gp^m) e_j.$$
In the above, we let $e_0$ denote a constant function.
As in the classical case of modular cusp forms, we have $\rho_j(\gp) \in [-2,2]$ and it is natural to study the distribution of $( \rho_j(\gp))$ in $[-2,2]$ as we vary the prime $\gp$ or $j \geq 1.$  This investigation was carried out by Imamoglu and Raulf \cite{IR}, who proved the following theorem:

\begin{thm}[{\cite[Main Theorem]{IR}}] \label{vertical-Bianchi}
Let $I$ be a subinterval of $[-2,2].$  For a fixed prime $\gp$ and $T > 0,$ we define
$$\lambda(T) = \#\{j:\,\lambda_j \leq T\}$$
and
$$N_I(\gp,T) := \#\{\lambda_j \leq T:\,\rho_j(\gp) \in I\}.$$
Then,
$$\lim_{T \to \infty} \frac{N_I(\gp,T)}{\lambda(T)} = \frac{1}{\pi}\int_I \frac{\left(1 + \frac{1}{N(\gp)}\right)}{\left(1 + \frac{1}{N(\gp)}\right)^2 - \frac{t^2}{N(\gp)}} \sqrt{1 - \frac{t^2}{4}} dt.$$
\end{thm}
The techniques in \cite{IR} can be extended to prove the following theorem:

\begin{thm}\label{average-ST-Bianchi}
Let $I$ be a subinterval of $[-2,2]$ and $\pi_K (x) = \# \{\gp :\,  N(\gp) \le x \}$.   For $j\geq 1,$ we define
$$N_I(j,x) = \#\{\gp:\,N(\gp) \leq x \text{ and }\rho_j(\gp) \in I\}.$$
Then,
$$\lim_{x \to \infty}\lim_{T \to \infty} \frac{\sum_{\lambda_j \leq T}N_I(j,x)}{\lambda(T)\pi_K(x)} = \frac{1}{\pi}\int_I  \sqrt{1 - \frac{t^2}{4}}dt.$$
\end{thm}
Following the theme of this article, we define the Hecke angles $\theta_j (\gp) \in [0,1]$ such that $\rho_j (\gp) = 2 \cos (\pi \theta_j (\gp))$. We investigate the pair correlation function of families of straightened Hecke angles $\theta_j (\gp) $. 
 The key tool in this investigation will be an explicit trace formula for Hecke operators acting on these spaces.   This formula was derived by Raulf \cite{Raulf}) and was used by Imamoglu and Raulf \cite{IR} to derive Theorem \ref{vertical-Bianchi}.  The following formula immediately follows from the work of \cite{Raulf} and \cite{IR}.
\begin{prop}[{\cite[Theorem 2.8]{Raulf} and \cite[Lemma 3.13]{IR}}] \label{trace-Bianchi} 
Let $\gp_1, \dots, \gp_h$ be distinct primes and let $\underline{m} = (m_1, \dots , m_h)$ be a tuple of non-negative integers. Let $\mathfrak a = \gp_1^{m_1} \cdots \gp_h^{m_h}$. Then, for any $T > 0$ 
$$
\sum_{\lambda_j \le T} \prod_{i=1}^h X_{m_i} (2 \cos \pi \theta_j (\gp_i))  = 
 \frac{\V}{8\pi^{3/2}} \delta_{2 | \underline{m}} \mathrm{Nm}(\mathfrak a)^{-1/2} T^{3/2} + O(T).
$$
\end{prop}

As before, we fix $0 < \psi < 1$ and an interval $\mathcal I_L$ and define
$$
\Lj := \Lj\,(x,L,\psi) := \#\left\{\gp \leq x: \,\theta_j (\gp) \in\mathcal I_L \right\},
$$
and
\begin{equation*}
R_{j, x,L}(s) := \frac{1}{\Lj} \#\left \{ 1 \leq \gp \neq \gq \leq x:
 \theta_j (\gp),\,\theta_j (\gq) \in \mathcal I_L, |H(\theta_j (\gp)) - H(\theta_j (\gq))|  \leq \frac{2s}{L\Lj} \right\}.
\end{equation*}
Finally, we define the { local pair correlation function around $\psi$} as
$$
R_{j,\psi}(s) := \lim_{x \to \infty} R_{j,x,L}(s).
$$

Just as in the previous sections, we derive the following theorem:
\begin{thm}\label{Bianchi-main}
Assume the Sato-Tate conjecture for the sequence $\{\theta_j (\gp)\}_{\mathrm{Nm}(\gp) \to \infty}$. Let $0 < \psi < 1$ and $L = L(x) \approx \log \log x.$ For any real number $s >0$, the limit 
$$
\lim_{x \to \infty \atop{ T \to \infty}} \frac{1}{ \lambda (T)} \sum_{\lambda_j \leq T} R_{j,x,L}(s) = 2s
$$
That is, 
$$
 \frac{1}{ \lambda (T)} \sum_{\lambda_j \leq T}  R_{j,\psi}(s) \sim 2s \text{ as } T \to \infty.
$$
\end{thm}

\begin{remark}
We note that unlike Theorems \ref{Theorem2} and \ref{hilbert-Theorem2}, where we had to assume strong growth conditions on the weights $k(x),$ we do not require any analogous growth conditions on $T.$  This is because the error terms in Raulf's trace formula (Proposition \ref{trace-Bianchi}) do not contain powers of $\mathrm{Nm} (\mathfrak a).$  This helps in controlling the exponential sums arising in the pair correlation sums for the families considered in this section.
\end{remark}

\bibliographystyle{amsalpha}
\bibliography{biblio}
\end{document}